\documentclass[12pt,reqno]{amsart}
\usepackage{amsmath}

\usepackage{mathrsfs}
\usepackage{cases}
\usepackage{txfonts}
\usepackage{amsfonts}
\usepackage{mathrsfs}
\usepackage{amssymb}
\usepackage{amsmath,amssymb}
\usepackage{color}
\usepackage{amsthm}

\newtheorem{theorem}{Theorem}
\newtheorem{lemma}{Lemma}

\newtheorem{corollary}{Corollary}
\newtheorem*{question}{Question}
\newtheorem{property}{Property}
 \theoremstyle{definition}
\newtheorem{definition}{Definition}

\newtheorem{remark}{Remark}
\numberwithin{equation}{section}
\begin{document}
\title[Critical $Q$ spaces ]%
{Analytic version of critical $Q$ spaces and their  properties}

\author{ Pengtao Li, junming liu \&  Zengjian Lou*}

\address{Department of Mathematics, Shantou University, Shantou, Guangdong,
515063, \newline P.~R.~China}
 \email{ptli@stu.edu.cn}
 \email{08jmliu@stu.edu.cn}
 \email{zjlou@stu.edu.cn}

\begin{abstract}
In this paper, we establish an analytic version of critical spaces
$Q_{\alpha}^{\beta}(\mathbb{R}^{n})$ on unit disc $\mathbb{D}$,
denoted by $Q^{\beta}_{p}(\mathbb{D})$. Further we prove a relation
between $Q^{\beta}_{p}(\mathbb{D})$ and Morrey spaces. By the
boundedness of two integral operators, we give the multiplier spaces
of $Q^{\beta}_{p}(\mathbb{D})$.
\end{abstract}
\thanks{This work was supported by NNSF(Grant No. 11171203 and No. 11201280); New Teacher's Fund for Doctor Stations, Ministry of Education No.20114402120003;
Guangdong Natural Science Foundation (Grant No. 10151503101000025
and No. S2011040004131); Foundation for Distinguished Young Talents
in Higher Education of Guangdong, China, LYM11063.}
\thanks{*Corresponding author.}

\maketitle
\vspace{0.1in}
\section{Introduction}
As a new space between $W^{1,n}(\mathbb{R}^{n})$ and
$BMO(\mathbb{R}^{n})$, $Q$-type space has been studied extensively
since 1990s. In 1995, on the unit disc $\mathbb{D}$ of the complex
plane, R. Aulaskari, J. Xiao and R. Zhao first introduced a class of
M\"{o}bius invariant analytic function space $Q_{p}(\mathbb{D})$,
$p\in (0,1)$. The class $Q_{p}(\mathbb{D})$, $p\in (0,1)$, can be
seen as subspaces and subsets of $BMOA$ and $UBC$ on $\mathbb{D}$.
Since then, many studies on $Q_{p}(\mathbb{D})$ and their
characterizations have been done. We refer the reader to
\cite{ASX}, \cite{AXZ}, \cite{NX}, \cite{XJ2} and the references
therein.

In order to generalize $Q_{p}(\mathbb{D})$ to $\mathbb{R}^{n}$, in
\cite{EJPX}, M. Essen, S. Janson, L. Peng and J. Xiao introduced a
class of $Q$-type spaces of several real variables
$Q_{\alpha}(\mathbb{R}^{n})$, $\alpha\in (0, 1)$, which is defined
as the set of all measurable functions with
$$\sup_{I}(l(I))^{2\alpha-n}\int_{I}\int_{I}\frac{|f(x)-f(y)|^{2}}{|x-y|^{n+2\alpha}}dxdy<\infty,$$
where the supremum is taken over all cubes $I$ with edge length
$l(I)$ and the edges parallel to the coordinate axes in
$\mathbb{R}^{n}$. Later, in \cite{DX}, G. Dafni and J. Xiao
established the Carleson measure characterization of
$Q_{\alpha}(R^{n})$. For more information of $Q_{\alpha}(R^{n})$ and
their applications, we refer the reader to \cite{DX}, \cite{EJPX}
and \cite{XJ2}. For the generalization of $Q_{\alpha}(R^{n})$, we
refer to \cite{LZ} and \cite{YY}.

It is easy to see that $Q-$type spaces has a  structure similar to
$BMO(\mathbb{R}^{n})$. Moreover, the definition of
$Q_{\alpha}(\mathbb{R}^{n})$ implies that the functions in
$Q_{\alpha}(\mathbb{R}^{n})$ should own some regularity. Hence, in
many cases, $Q_{\alpha}(\mathbb{R}^{n})$ can be  seen as an adequate
replacement of $BMO(\mathbb{R}^{n})$. In recent years, these spaces
have been applied to the study of partial differential equations by
several authors. In \cite{XJ0}, J. Xiao got the well-posedness of
the Navier-Stokes equation with data in
$Q_{\alpha}^{-1}(\mathbb{R}^{n})$, where
$Q_{\alpha}^{-1}(\mathbb{R}^{n})=\nabla \cdot
(Q_{\alpha}(\mathbb{R}^{n}))^{n}$. For $\alpha\in (-\frac{n}{2},0)$,
$Q_{\alpha}(\mathbb{R}^{n})=BMO(\mathbb{R}^{n})$. Hence Xiao's
result generalizes that of Koch and Tataru \cite{KT}. Inspired the
results of \cite{KT} and \cite{XJ0}, P. Li and Z. Zhai introduced a
new critical $Q$-type spaces $Q_{\alpha}^{\beta}(\mathbb{R}^{n})$
with $\alpha>0$ and $\max\{\frac{1}{2}, \alpha\}<\beta<1$ such that
$\alpha+\beta-1\geq 0$. With data in $Q_{\alpha}^{\beta,
-1}(\mathbb{R}^{n})=\nabla\cdot
(Q_{\alpha}^{\beta}(\mathbb{R}^{n}))^{n}$ being small, in \cite{LZ},
P. Li and Z. Zhai obtained the well-posedness of the generalized
Navier-Stokes equations. This result generalizes the well-posed
results of \cite{KT} and \cite{XJ0} to the fractional cases. Also,
the spaces $Q^{\beta}_{\alpha}(R^{n})$ can be used to study the
well-posedness of quasi-geostrophic equation, see \cite{LZ2}.

Now we state the motivation of this paper. In \cite{LZ}, the authors
introduced the following $Q$-type spaces to study the generalized
Navier-Stokes equations.

\begin{definition}\label{def1}
For $0<\alpha<1$, $\max\{\alpha, \frac{1}{2}\}<\beta<1$ and
$\alpha+\beta\geq 1$, $Q_{\alpha}^{\beta}(\mathbb{R}^{n})$ is
defined as the  space of all measurable functions with
$$\sup_{I}(l(I))^{2\alpha-n+2\beta-2}\int_{I}\int_{I}\frac{|f(x)-f(y)|^{2}}{|x-y|^{n+2\alpha-2\beta+2}}dxdy<\infty,$$
where the supremum is taken over all cubes $I$ with edge length
$l(I)$ and the edges parallel to the coordinate axes in
$\mathbb{R}^{n}$.
\end{definition}

By a simple computation, we can see that
$Q^{\beta}_{\alpha}(\mathbb{R}^{n})$ is the spaces which is
invariant under the dilation:
$f_{\lambda}(x)=\lambda^{2\beta-1}f(\lambda x)$. In the research of
partial differential equations, such spaces is called critical
spaces and play an important role in the well-posedness of fluid
equations. If $\beta=1$, $Q_{\alpha}^{\beta}(\mathbb{R}^{n})$
becomes $Q_{\alpha}(\mathbb{R}^{n})$  introduced in \cite{EJPX}. It
is well-known that when $n=1$, $Q_{\alpha}(\mathbb{R}^{1})$ can be
seen as the boundary value of $Q_{p}(\mathbb{D})$. It is natural to
ask
\begin{question}
What about the analytic version of
$Q_{\alpha}^{\beta}(\mathbb{R}^{n})$?
\end{question}

The aim of this paper is to establish an analytic version on the
unit disc of $Q_{\alpha}^{\beta}(\mathbb{R}^{n})$. Precisely, we
prove a Carleson-measure characterization of the spaces
$Q_{p}^{\beta}(\mathbb{T})$, the $Q$ type spaces on the unit circle
associate to $Q_{\alpha}^{\beta}(\mathbb{R}^{n})$.
 Using this characterization implies that
 $Q^{\beta}_{p}(\mathbb{D})$, the space of analytic functions in $\mathbb{\mathbb{D}}$ which
 are the Poisson integral of $Q_{p}^{\beta}(T)$ functions, can be seen as a harmonic
extension of $Q^{\beta}_{p}(\mathbb{T})$. Proofs of these results
are motivated by the idea of \cite{XJ2} by J. Xiao. Then we study
properties of the spaces $Q^{\beta}_{p}(\mathbb{D})$. By the
$\nu-$derivative of $f$ on $\mathbb{D}$, we establish a relation
between $Q^{\beta}_{p}(\mathbb{D})$ and Morrey spaces
$\mathcal{L}^{2, \lambda}$. For the special case $\beta=1$, such
relation has been obtained by Z. Wu and C. Xie in \cite{WX}. Hence
our result can be seen as a generalization of that in \cite{WX}.
Finally, by the boundedness of two integral operators, we obtain the
multiplier spaces of $Q^{\beta}_{p}(\mathbb{D})$.

This paper is organized as follows. In Section 2, we state some
notations and terminology which will be used in the sequel. In
Section 3, we investigate the harmonic extension of
$Q_{\alpha}^{\beta}(\mathbb{T})$ and the analytic version of
$Q_{\alpha}^{\beta}(\mathbb{R}^{n})$. The relation between
$Q^{\beta}_{p}(\mathbb{D})$ and Morrey spaces are given in Section
4. Section 5 is devoted to the study of multiplier spaces of
$Q_{p}^{\beta}(\mathbb{D})$.

Throughout this paper, for two functions $F$ and $G$, we say
$F\lesssim G$ if there is a positive constant $C$ independent of $F$
and $G$ such that $F\leq CG$. Furthermore, we say $F\approx G$ (that
is, $F$ is comparable with $G$ ) whenever $F\lesssim G\lesssim F$.

\section{Notation and preliminaries}
 Let $\mathbb{D}$ and  $\mathbb{T}$ denote respectively the open unit disc and the unit circle in the
complex plane $\mathbb{C}$. Let $H(\mathbb{D})$ be the space of all
analytic functions on $\mathbb{D}$. For $0<p<\infty$,
$Q_{p}(\mathbb{D})$ is the set of all functions $f\in H(\mathbb D)$
with
$$\|f\|_{Q_{p}(\mathbb{D})}=|f(0)|+\sup_{a\in\mathbb D}\left(\int_{\mathbb{D}}|f'(z)|^{2}(1-|\sigma_{a}(z)|^2)^{p}dA(z)\right)^{1/2}<\infty,$$
where $\sigma_{a}(z)=\frac{a-z}{1-\overline{a}z}$ is the M\"{o}bius
map and $dA(z)=\frac{1}{\pi}dxdy$ is the normalized Lebesgue area
measure. If $p=1$, then $Q_{p}(\mathbb{D})=BMOA(\mathbb{D})$, the
John-Nirenberg space of all analytic functions with bounded mean
oscillation. That is $f\in BMOA(\mathbb{D})$ if and only if $f$
belongs to Hardy space $H^{2}(\mathbb{D})$ and satisfies
$$\|f\|_{BMOA}=\sup_{I\subset T}|I|^{-1}\int_{I}|f(\zeta)-f_{I}|\frac{|d\zeta|}{2\pi}<\infty,$$
where the supremum is taken over all open subarcs $I\subset
\mathbb{T}$ with
$$|I|=\int_{I}\frac{|d\zeta|}{2\pi},\ \mbox{and}\ f_{I}=|I|^{-1}\int_{I}f(\zeta)\frac{|d\zeta|}{2\pi}.$$

For $-\infty<p<\infty$, $Q_{p}(\mathbb{T})$ is the space of all
Lebesgue measurable functions $f: \mathbb{T}\rightarrow \mathbb{C}$
with
$$\|f\|_{Q_{p}(\mathbb{T})}=\sup_{I\subset \mathbb{T}}\left( |I|^{-p}\int_{I}\int_{I}\frac{| f(\zeta)-f(\eta)|^{2}}{|\zeta-\eta|^{2-p}}
|d\eta| |d\zeta| \right)^{1/2}<\infty,$$ where the supremum is taken
over all arcs $I\subset \mathbb{T}$.  For more information on spaces
$Q_{p}(\mathbb{D})$ and $Q_{p}(\mathbb{T})$, see \cite{NX},
\cite{XJ1}, \cite{XJ2} and \cite{XJ3}, for example.

Let $I\subset \mathbb{T}$ be an interval. We define Carleson box as
$$S(I)=\Big\{z\in \mathbb{D}: 1-|I|\leq |z|<1, \frac{z}{|z|}\in I\Big\}.$$
For $0<p<\infty$, a positive Borel measure $\mu$ on $\mathbb{D}$ is
said to be a $p$-Carleson measure if
$$\sup_{I\subset T}\frac{\mu(S(I))}{|I|^{p}}<\infty.$$
If $p=1$, $p$-Carleson measure is the classical Carleson measure,
see \cite{CL} and \cite{GJB}.

The following result on $p$-Carleson measure is well-known, see
\cite{ASX}.

\begin{lemma}\label{lem1}
For $0<p<\infty$, a positive Borel measure $\mu$ on $\mathbb{D}$ is
a $p$-Carleson measure if and only if
$$\sup_{a\in \mathbb{D}}\int_{\mathbb{D}}\left(\frac{1-|a|^{2}}{|1-\overline{a}z|^{2}}\right)^{p} d\mu(z)<\infty.$$
\end{lemma}

\section{\bf Harmonic extension}
In this section, we establish an analytic version on the unit disc
of $Q_{\alpha}^{\beta}(\mathbb{R}^{n})$. At first, we introduce the
definition of $Q-$type spaces on unit circle associated with
$Q_{\alpha}^{\beta}(\mathbb{R}^{n})$.
\begin{definition}
Let $0<p<1$, $1/2<\beta<1$ and $f\in L^{2}(\mathbb{T})$. We say
$f\in Q_{p}^{\beta}(\mathbb{T})$ if
$$ \|f\|_{Q_{p}^{\beta}(\mathbb{T})}=\sup_{I\subset \mathbb{T}}\left(|I|^{2\beta-2-p}\int_{I} \int_{I}
\frac{|f(\zeta)-f(\eta)|^{2}}{|\zeta-\eta|^{4-p-2\beta}}|d\zeta||d\eta|
\right)^{1/2}<\infty,$$ where the supremum is taken over all arcs
$I\subset \mathbb{T}$.
\end{definition}
That $Q_{p}^{\beta}(\mathbb{T})$ can be seen as the subspaces of the
following $BMO$ type spaces.

\begin{definition} Let $1/2<\beta<1$. Define
$BMO^{\beta}(\mathbb{T})$ as the space of $f\in L^{2}(\mathbb{T})$
with
$$\|f\|_{BMO^{\beta}(\mathbb{T})}^{2}=\sup_{I\subset \mathbb{T}}|I|^{4\beta-5}\int_{I}|f(e^{is})-f_{I}|^{2}ds<\infty,$$
where the supremum is taken over all arcs $I\subset \mathbb{T}$.
\end{definition}

The following properties can be deduced from the definitions of
$BMO^{\beta}(\mathbb{T})$ and $Q_{p}^{\beta}(\mathbb{T})$
immediately (we refer the reader to Theorem 3.2 of \cite{LZ}).

\begin{property}
Given $p\in(0,1)$ and $\beta\in(1/2,1)$. Then
\begin{itemize}
\item[(i)] $Q^{\beta}_{p}(\mathbb{T})$ is increasing in $p$ for a fixed $\beta$, i.e.
$$Q^{\beta}_{p_{1}}(\mathbb{T})\subseteq Q^{\beta}_{p_{2}}(\mathbb{T}), \text{ if } p_{1}\leq p_{2};$$
\item[(ii)] $Q^{\beta}_{p}(\mathbb{T})\subseteq BMO^{\beta}(\mathbb{T}).$
\end{itemize}
\end{property}

Let
$d\mu_{z}(\zeta)=\frac{1-|z|^{2}}{|\zeta-z|^{2}}\frac{|d\zeta|}{2\pi}$
be the Poisson measure on $\mathbb{T}$. In addition, denote
$$\hat{f}(z)=\int_{\mathbb{T}}f(\zeta)d\mu_{z}(\zeta),\ \ z\in \mathbb{D}.$$
Similar to Lemma 7.1.1 of \cite{XJ2}, we can obtain the following lemma.
\begin{lemma}\label{le-main}
Given $p\in(0,1)$, $\beta\in(\frac{1}{2}, 1)$ and $p+2\beta>2$. Let
$I$ and $J$ be two arcs on $\mathbb{T}$ centered at
$\zeta_{0}=e^{is_{0}}$ with $|J|\geq 3|I|$. If $f\in
L^{2}(\mathbb{T})$, then
\begin{equation}
\begin{split}\nonumber
&\int_{S(I)}|\nabla
\hat{f}(z)|^{2}(1-|z|^{2})^{p-2+2\beta}dm(z)\\
\lesssim&
\int_{J}\int_{J}\frac{|f(e^{it})-f(e^{is})|^{2}}{|e^{it}-e^{is}|^{4-p-2\beta}}dtds
+|I|^{2\beta+p}\left(\int_{|t|\geq
2|J|/3}\frac{|f(e^{i(t+s_{0})})-f_{J}|}{t^{2}}dt\right)^{2},
\end{split}
\end{equation}
where $\nabla=(2\partial/\partial z, 2\partial
/\partial\overline{z})$ stands for the gradient vector.
\end{lemma}
\begin{proof}
Without loss of generality, assume that $\zeta_{0}=1$ and $\phi$ is
a function with $0\leq \phi\leq 1$ such that
$$\begin{cases}
\phi=1\text{ on }\frac{J}{3};\\
\mbox{ supp }\phi\subset \frac{2J}{3};\\
|\phi(z)-\phi(w)|\lesssim \frac{|z-w|}{|J|}\text{ for all }z, w\in
T.
\end{cases}$$
 Writing
$\psi=1-\phi$, we then have
$$f=f_{J}+(f-f_{J})\phi+(f-f_{J})\psi=:f_{1}+f_{2}+f_{3}.$$
Note that $|\nabla \hat{f_{J}}|=0$, since $f_{1}$ is constant. So
$|\nabla \hat{f}|^{2}$ is dominated by $|\nabla \hat{f_{3}}|^{2}$
and $|\nabla \hat{f_{2}}|^{2}$. For $z=re^{i\theta}$ in the Carleson
box $S(I)$,
 \begin{equation}
 \begin{split}\nonumber
|\nabla \hat{f_{3}}(re^{i\theta})|&\lesssim
\int_{0}^{2\pi}\frac{|f_{3}(e^{it})|}{(1-r)^{2}+(\theta-t)^{2}}dt
\\&
\lesssim \int_{|t|>2|J|/3}|f(e^{it})-f_{J}|\frac{dt}{t^{2}}.
\end{split}
\end{equation}
So we get
 \begin{equation}
 \begin{split}\nonumber
&\int_{S(I)}|\nabla
\hat{f_{3}}(z)|^{2}(1-|z|^{2})^{p-2+2\beta}dA(z)\\
\lesssim& \int_{S(I)} \left( \int_{|t|>2|J|/3}|f(e^{it})-f_{J}|\frac{dt}{t^{2}} \right)^{2}(1-|z|^{2})^{p-2+2\beta}dA(z)\\
\lesssim& \left( \int_{|t|>2|J|/3}|f(e^{it})-f_{J}|\frac{dt}{t^{2}}
\right)^{2} \int_{0}^{|I|/2}d\theta\int_{0}^{|I|}r^{p-2+2\beta}rdr\\
\lesssim &|I|^{p+2\beta}\left(\int_{|t|>2|J|/3}|f(e^{it})-f_{J}|\frac{dt}{t^{2}} \right)^{2}.
\end{split}
\end{equation}

Now for the integral over $S(I)$ of $|\nabla \hat{f_{2}}|^{2}$, by Lemma 6.1.1 of \cite{XJ2},
 \begin{equation}
 \begin{split}\nonumber
\int_{S(I)}&|\nabla \hat{f_{2}}(z)|^{2}(1-|z|^{2})^{p-2+2\beta}dA(z)
\\ &\lesssim
\int_{\mathbb{T}}\int_{\mathbb{T}}\frac{|f_{2}(\zeta)-f_{2}(\eta)|^{2}}{|\zeta-\eta|^{4-p-2\beta}}|d\zeta||d\eta|\\
 & \lesssim\Big(\int_{\zeta\in J, \eta\in J}+\int_{\zeta \not\in J, \eta\in
\frac{3J}{4}}+\int_{\eta\in J, \zeta \in \frac{3J}{4}}\Big)
 \frac{|f_{2}(\zeta)-f_{2}(\eta)|^{2}}{|\zeta-\eta|^{4-p-2\beta}}|d\zeta||d\eta|\\
&=:M_{1}+M_{2}+M_{3}.
\end{split}
\end{equation}
Since
\begin{equation}
 \begin{split}\nonumber
|f_{2}(\zeta)-f_{2}(\eta)|&=|(f(\zeta)-f_{J})\phi(\zeta)-(f(\eta)-f_{J})\phi(\eta)|
\\&= |(f(\zeta)-f(\eta))\phi(\zeta)+f(\eta)\phi(\zeta)-f_{J}\phi(\eta)-(f(\eta)-f_{J})\phi(\eta)|
\\&
\lesssim
|f(\zeta)-f(\eta)||\phi(\zeta)|+|f(\eta)-f_{J}||\phi(\zeta)-\phi(\eta)|
\\&
\lesssim |f(\zeta)-f(\eta)|+|J|^{-1}|\zeta-\eta||f(\eta)-f_{J}|,
\end{split}
\end{equation}
we have
\begin{equation}
 \begin{split}\nonumber
M_{1}&=\int_{J}\int_{J}\frac{|f_{2}(\eta)-f_{2}(\zeta)|^{2}}{|\zeta-\eta|^{4-p-2\beta}}|d\zeta||d\eta|
\\&\lesssim \int_{J}\int_{J}\frac{|f(\eta)-f(\zeta)|^{2}}{|\zeta-\eta|^{4-p-2\beta} }|d\zeta||d\eta|
 +\int_{J}\int_{J}\frac{1}{|J|^{2}}\frac{|\zeta-\eta|^{2}|f(\eta)-f_{J}|^{2}}{|\zeta-\eta|^{4-p-2\beta}}
\\&\lesssim
\int_{J}\int_{J}\frac{|f(\eta)-f(\zeta)|^{2}}{|\zeta-\eta|^{4-p-2\beta}
}|d\zeta||d\eta| +
\frac{1}{|J|^{2}}\int_{J}\int_{J}|f(\eta)-f_{J}|^{2}|d\eta||\zeta-\eta|^{p-2+2\beta}|d\zeta|
\\&\lesssim
\int_{J}\int_{J}\frac{|f(\eta)-f(\zeta)|^{2}}{|\zeta-\eta|^{4-p-2\beta}
}|d\zeta||d\eta|
 +\frac{1}{|J|^{2}}\int_{J}|f(\eta)-f_{J}|^{2}|J|^{p-1+2\beta}|d\eta|.
\end{split}
\end{equation}
Note that
\begin{equation}
 \begin{split}\nonumber
\frac{1}{|J|^{2}}\int_{J}|f(\eta)-f_{J}|^{2}|J|^{p-1+2\beta}|d\eta|
=&\frac{1}{|J|^{3-p-2\beta}}\int_{J}|f(\eta)-f_{J}|^{2}|d\eta|\\
\lesssim& \frac{1}{|J|^{4-p-2\beta}}\int_{J}\int_{J}|f(\zeta)-f(\eta)|^{2}|d\zeta||d\eta|\\
\lesssim&\int_{J}\int_{J}\frac{|f(\zeta)-f(\eta)|^{2}}{|\zeta-\eta|^{4-p-2\beta}}|d\zeta||d\eta|,
\end{split}
\end{equation}
where we used a fact that
$$f(\eta)-f_{J}=\frac{1}{|J|}\int_{J}(f(\eta)-f(\zeta))d\zeta$$
and $|\zeta-\eta|\leq 2|J|$. We obtain
$$M_{1}\lesssim \int_{J}\int_{J}\frac{|f(\zeta)-f(\eta)|^{2}}{|\zeta-\eta|^{4-p-2\beta}}|d\zeta||d\eta|.$$

Now, we estimate $M_{2}$. Recall that
$$M_{2}=\int_{\zeta\not\in J}\int_{\eta\in \frac{3J}{4}}\frac{|f_{2}(\zeta)-f_{2}(\eta)|}{|\zeta-\eta|^{2-p+2\beta}}|d\zeta||d\eta|.$$
Since $\zeta\not\in J$, then $\phi(\zeta)=0$ and
$f_{2}(\zeta)=(f(\zeta)-f_{J})\phi(\zeta)=0$. This implies that
\begin{equation}
\begin{split}\nonumber
M_{2}&=\int_{\zeta\not\in J}\int_{\eta\in
\frac{3J}{4}}\frac{|f_{2}(\eta)|}{|\zeta-\eta|^{2-p+2\beta}}|d\zeta||d\eta|
\\&=\int_{\eta\in \frac{3J}{4}}|f_{2}(\eta)|^{2}|d\eta|\int_{\zeta\not\in
J}\frac{|d\zeta|}{|\zeta-\eta|^{4-p-2\beta}}.
\end{split}
\end{equation}
Assume $J=\{e^{i\theta}:\ |\theta-s_{0}|\leq h \}$. Since
$e^{is_{0}}=1$, then $s_{0}=0$. So $|J|=\frac{h}{2\pi}$ and
$|\theta|\leq 2\pi J$. Take $\zeta=e^{i\theta}\not \in J$. Then
$|\theta|>2\pi J$. If $\eta\in \frac{3J}{4}$, then
$\eta=e^{i\arg\eta}$ and $|\arg \eta|\leq\frac{3\pi}{2}|J|$. This
leads that $|\zeta-\eta|\geq \frac{\pi}{2}|J|$. Hence
\begin{equation}
\begin{split}\nonumber
M_{2}&\lesssim \int_{\eta\in
\frac{3J}{4}}|f_{2}(\eta)-f_{J}|^{2}\frac{|d\eta|}{|J|^{3-p-2\beta}}
\\&\lesssim \frac{1}{|J|^{4-p-2\beta}}\int_{\zeta\not\in
J}\int_{\eta\in \frac{3J}{4}}|f(\zeta)-f(\eta)|^{2}|d\zeta||d\eta|.
\end{split}
\end{equation}
For $M_{3}$, the proof is similarly and is so omitted.
\end{proof}

\begin{theorem}\label{th1}
Given $0<p<1$, $\beta\in(1/2, 1)$ and $p-2+2\beta>0$. Let $f\in BMO^{\beta}(\mathbb{T})$.
The following conditions are equivalent:
\begin{itemize}
\item[(1)]  $f\in Q_{p}^{\beta}(\mathbb{T})$;
\item [(2)] $$\sup_{I\subset \mathbb{T}}|I|^{2\beta-p-2}\int_{0}^{|I|}\left(\int_{I}|f(e^{i(s+t)})-f(e^{is})|^{2}ds\right)\frac{dt}{t^{4-p-2\beta}}<\infty,$$
where the supremum is taken over all arcs $I\subseteq \mathbb{T};$
\item[(3)] $|\nabla\hat{f}(z)|^{2}(1-|z|^{2})^{p-2+2\beta}dA(z)$
is a $(p+2-2\beta)$-Carleson measure.
\end{itemize}
\end{theorem}
\begin{proof}
We give the proof according the following order
(3)$\Rightarrow$(2)$\Rightarrow$(1)$\Rightarrow$(3).

 (3)$\Rightarrow$(2). Without loss of generality, assume that $I$ is
 the
interval $(0, |I|)$ with $|I|<\frac{1}{4}$. If $t\in (0, |I|)$, then
\begin{equation}
\begin{split}\nonumber
\left(\int_{3I}|f(e^{i(v+t)})-f(e^{iv})|^{2}dv\right)^{1/2}&\leq
2\int_{1-t}^{1}\left( \int_{4I}\Big|\frac{\partial \hat{f}}{\partial
n}(ue^{is})\Big|^{2}ds\right)^{1/2}du \\
&+ 2\int_{0}^{t}\left( \int_{4I}\Big|\frac{\partial
\hat{f}}{\partial
\theta}((1-t)e^{is})\Big|^{2}ds\right)^{1/2}du\\
&=:I_{1}+I_{2}.
\end{split}
\end{equation}
Since
\begin{equation}
\begin{split}\nonumber
\int_{0}^{|I|}\frac{I_{1}^{2}}{t^{4-p-2\beta}}dt&\lesssim
\int_{0}^{|I|}t^{p+2\beta-2}\left(\int_{4I}\Big|\frac{\partial
\hat{f}}{\partial n}((1-t)e^{is})\Big|^{2}ds \right)dt
\\&\lesssim \int_{S(4I)}|\nabla
\hat{f}(z)|^{2}(1-|z|)^{p+2\beta-2}dA(z),
\end{split}
\end{equation}
and
\begin{equation}
\begin{split}\nonumber
\int_{0}^{|I|}\frac{I_{2}^{2}}{t^{4-p-2\beta}}dt&\lesssim
\int_{0}^{|I|}t^{p+2\beta-2}\left(\int_{4I}\Big|\frac{\partial
\hat{f}}{\partial \theta}(re^{is})\Big|^{2}ds \right)dt
\\&\lesssim \int_{S(4I)}|\nabla
\hat{f}(z)|^{2}(1-|z|)^{p+2\beta-2}dA(z).
\end{split}
\end{equation}
We have
 \begin{equation}
\begin{split}\nonumber
&\frac{1}{|I|^{p+2-2\beta}}\int_{0}^{|I|}
\left(\int_{I}|f(e^{i(\theta+t)})-f(e^{i\theta})|^{2}d\theta\right)\frac{dt}{t^{4-p-2\beta}}\\
\lesssim&\frac{1}{|I|^{p+2-2\beta}}\int_{0}^{|I|}\int_{3I}
|f(e^{i(\theta+t)})-f(e^{i\theta})|^{2}d\theta\frac{dt}{t^{4-p-2\beta}}\\
\lesssim&\frac{1}{|I|^{p+2-2\beta}}\int_{0}^{|I|}
\Big(I_{1}^{2}+I_{2}^{2}\Big)\frac{dt}{t^{4-p-2\beta}}\\
\lesssim& \frac{1}{|I|^{p+2-2\beta}}\int_{S(4I)}|\nabla
\hat{f}(z)|^{2}(1-|z|)^{p+2\beta-2}dA(z)\leq C,
\end{split}
\end{equation}
where we used a fact that $|\nabla
\hat{f}(z)|^{2}(1-|z|)^{p+2\beta-2}dA(z)$ is a
$(p+2-2\beta)$-Carleson measure.

(2)$\Rightarrow$(1). Without loss of generality, we assume that $I$
is a small subarc of $\mathbb{T}$ (say $|I|\leq 1/4$) and assume
that $I=(a, b)\subset [0, 2\pi]$.  We get
 \begin{equation}
\begin{split}\nonumber
\int_{I}\int_{I}\frac{|f(e^{is})-f(e^{it})|^{2}}{|e^{is}-e^{it}|^{4-p-2\beta}}dsdt
&\lesssim
\int_{a}^{b}\int_{a<s+t<b}\frac{|f(e^{i(s+t)}-f(e^{is}))|^{2}}{|t|^{4-p-2\beta}}dtds\\
&=\int_{a}^{b}\left(
\int_{a-s}^{0}+\int_{0}^{b-s}\right)\frac{|f(e^{i(s+t)}-f(e^{is}))|^{2}}{|t|^{4-p-2\beta}}dtds\\
 &\lesssim \int_{0}^{b-a}\frac{1}{\theta^{4-p-2\beta}}\left(\int_{I}|f(e^{i(s+\theta)}-f(e^{is}))|^{2}
ds\right)d\theta\\
 &\lesssim |I|^{p+2-2\beta}.
\end{split}
\end{equation}
Hence $f\in Q_{p}^{\beta}(\mathbb{T})$.

(1)$\Rightarrow$(3). Let $e^{is}, e^{it}\in I$, then
$|e^{is}-e^{it}|\leq |I|$. We have
 \begin{equation}
\begin{split}\nonumber
\int_{I}\int_{I}|f(e^{is})-f(e^{it})|^{2}dsdt&\leq
|I|^{4-p-2\beta}\int_{I}\int_{I}\frac{|f(e^{is})-f(e^{it})|^{2}}{|e^{is}-e^{it}|^{4-p-2\beta}}dsdt
\\& \lesssim |I|^{6-4\beta} \|f\|_{Q_{p}^{\beta}(\mathbb{T})}^{2}.
\end{split}
\end{equation}
So
 \begin{equation}
\begin{split}\nonumber
\int_{I}|f(e^{is})-f_{I}|ds&\leq
\int_{I}\frac{1}{|I|}\int_{I}|f(e^{is})-f(e^{it})|dtds
\\ &\leq \frac{1}{|I|}\int_{I}\int_{I}|f(e^{is})-f(e^{it})|^{2}dtds
\\&\lesssim |I|^{5-4\beta}\|f\|_{Q_{p}^{\beta}(\mathbb{T})}^{2}.
\end{split}
\end{equation}

Let $J=3I$ and $|I|<1/3$. Using Lemma \ref{le-main} (for
$\zeta_{0}=0$), we only need to show
\begin{align}\label{eq3.1}
\int_{|t|\geq
2|J|/3}|f(e^{i(t+s_{0})})-f_{J}|\frac{dt}{t^{2}}\lesssim
|I|^{1-2\beta}\|f\|_{BMO^{\beta}(T)}.
\end{align}
It is easy to see that

\begin{align}\label{eq3.2}
\int_{|t|\geq |J|/3}|f(e^{it})-f_{J}|\frac{dt}{t^{2}}&\leq
\sum_{k=1}^{\infty}\int_{3^{k-1}|J|<|t|\leq
3^{k}|J|}|f(e^{it})-f_{J}|\frac{dt}{t^{2}}\nonumber
\\ &\lesssim \sum_{k=1}^{\infty}\frac{1}{(3^{k}|J|)^{2}}\int_{|t|\leq
3^{k}|J|}|f(e^{it})-f_{J}|dt\nonumber
\\&\lesssim\sum_{k=1}^{\infty}\frac{1}{(3^{k}|J|)^{2}}\int_{|t|\leq
3^{k}|J|}|f(e^{it})-f_{3^{k+1}J}|dt
+\sum_{k=1}^{\infty}\frac{|f_{3^{k+1}J}-f_{J}|}{3^{k}|J|}\nonumber
\\&=: M_{1}+M_{2}.
\end{align}

At first, we estimate the term $M_{1}$. Since
\begin{align}\label{eq3.3}
M_{1}&=\sum_{k=1}^{\infty}\frac{1}{(3^{k}|J|)^{2}}\int_{|t|\leq
3^{k}|J|}|f(e^{it})-f_{3^{k+1}J}|dt\nonumber
\\ &\lesssim \sum_{k=1}^{\infty}\frac{1}{(3^{k}|J|)}\left( \frac{1}{(3^{k}|J|)}\int_{|t|\leq
3^{k}|J|}|f(e^{it})-f_{3^{k+1}J}|^{2}dt\right)^{1/2}\nonumber
\\
&=\sum_{k=1}^{\infty}(3^{k}|J|)^{1-2\beta}\|f\|_{BMO^{\beta}(\mathbb{T})}\nonumber\\
&\lesssim |J|^{1-2\beta}\|f\|_{BMO^{\beta}(\mathbb{T})}.
\end{align}

We next deal with the term $M_{2}$. By the definition of
$BMO^{\beta}(\mathbb{R}^{n})$, for any $I$, $J\subset \mathbb{T}$
with $I\subset J$ and $|J|\leq 3|I|$,
\begin{equation}\nonumber
\begin{split}
|f_{I}-f_{J}|&\lesssim \frac{1}{|I|}\int_{I}|f(e^{is})-f_{J}|ds\\
&\lesssim\Big(\frac{1}{|I|}\int_{I}|f(e^{is})-f_{J}|^{2}ds\Big)^{1/2}\\
&\lesssim|J|^{2-2\beta}\|f\|_{BMO^{\beta}(\mathbb{T})}.
\end{split}
\end{equation}
Since $\beta>1/2$, we obtain
\begin{align}\label{eq3.4}
M_{2}=&\sum_{k=1}^{\infty}\frac{|f_{3^{k+1}J}-f_{J}|}{3^{k}|J|}\nonumber\\
\lesssim&\sum_{k=1}^{\infty}\frac{1}{3^{k}|J|}\Big(|f_{3^{k+1}J}-f_{3^{k}J}|+|f_{3^{k}J}-f_{3^{k-1}J}|\cdots
+|f_{3J}-f_{J}|\Big)\nonumber\\
\lesssim&\sum_{k=1}^{\infty}\frac{k+1}{3^{k}|J|}(3^{k+1}|J|)^{2-2\beta}\|f\|_{BMO^{\beta}(\mathbb{T})}\nonumber\\
\lesssim&|I|^{1-2\beta}\|f\|_{BMO^{\beta}(\mathbb{T})}.
\end{align}

So $(\ref{eq3.1})$ follows from $(\ref{eq3.2})$, $(\ref{eq3.3})$ and
$(\ref{eq3.4})$. The proof of this theorem is completed by Lemma
\ref{le-main}.
\end{proof}

Let $C^{1}(\mathbb{D})$ be the space of complex value functions
which is continuously differentiable on $\mathbb{D}$. Then we have
the following corollary.

\begin{corollary}
Given $0<p<1$, $1/2<\beta<1$ and $p+2\beta>2$. Let $F$ be a function
defined on $\overline{\mathbb{D}}$ such that $F\in
C^{1}(\mathbb{D})$ and $F|_{\mathbb{T}}=f$. If $|\nabla
F(z)|^{2}(1-|z|^{2})^{p}dm(z)$ is a $(p+2-2\beta)$-Carleson measure,
then $f\in Q_{p}^{\beta}(\mathbb{T})$.
\end{corollary}

 It is easy to see that $|\nabla \hat{f}|^{2}=4|f'(z)|$ when $f\in H(\mathbb{D})$.  Theorem \ref{th1}
suggests us to define $Q_{p}^{\beta}(\mathbb{D})$ spaces as
following.

\begin{definition}
 For $0<p<1$ and
$1/2<\beta\leq 1$. The $Q_{p}^{\beta}(\mathbb{D})$ stands for the
space of all $f\in H(\mathbb{D})$ satisfying
$$\sup_{|I|\subset \mathbb{T}}\frac{1}{|I|^{p+2-2\beta}}\int_{S(I)}|f'(z)|^{2}(1-|z|^{2})^{p-2+2\beta}dA(z)<\infty.$$
\end{definition}

Theorem \ref{th1} gives the answer to the Question. In fact, the
$Q_{p}^{\beta}(\mathbb{D})$ is the space of those analytic functions
in $\mathbb{D}$ which are the Poisson integral of a function $\psi$
with $\psi \in Q_{p}^{\beta}(\mathbb{T})$ (for related result, see
Theorem 3.1 of \cite{DP} and Proposition 5.1 of \cite{GD}).

\begin{remark}
In fact, by Lemma 1, we can prove that $Q_{p}^{\beta}(\mathbb{D})$
can be defined equivalently as
$$\|f\|_{Q_{p}^{\beta}(\mathbb{D})}=|f(0)|+\sup_{a\in
\mathbb{D}}\left(\int_{\mathbb{D}}|f'(z)|^{2}(1-|z|^{2})^{4\beta-4}(1-|\sigma_{a}(z)|^{2})^{p+2-2\beta}dA(z)
\right)^{1/2}.$$
\end{remark}

\section{\bf  $Q_{p}^{\beta}(\mathbb{D})$ and Morrey spaces }
In this section, we establish a relation between
$Q^{\beta}_{p}(\mathbb{D})$ and Morrey spaces. The Morrey spaces are
defined as follows.
\begin{definition}
Given $0< \lambda\leq 1$. The Morrey space $\mathcal{L}^{2,
\lambda}(\mathbb{D})$ is defined as the set of all $f$ which belong
to Hardy space $ H^{2}(\mathbb{D})$ and satisfy
$$\sup_{I\subset \mathbb{T}}\frac{1}{|I|^{\lambda}}\int_{I}|f(\zeta)-f_{I}|^{2}\frac{|d\zeta|}{2\pi}< \infty,$$
where the supremum is taken over all open subarcs $I\subset
\mathbb{T}$.
\end{definition}

Clearly, $\mathcal{L}^{2, 1}(\mathbb{D})= BMOA(\mathbb{D})$.

In \cite{WZ}, Z. Wu proved a Carleson measure characterization of
Morrey spaces. A similar result on the half plane can be found in
\cite[Theorem 1]{WX}.

\begin{lemma}\label{lem3}{\rm(Lemma 4.1 of \cite{WZ})}
Let $0<\lambda\leq 1$ and $f\in H(\mathbb{D})$. Then $f$ belongs to
$\mathcal{L}^{2, \lambda}(\mathbb{D})$ if and only if
$|f'(z)|^{2}(1-|z|^{2})dA(z)$ is a $\lambda$-Carleson measure.
\end{lemma}

For  the proof of the main result of this section, we need the
following technical lemma which can be compared with Theorem 3 of
\cite{WX}. For a fixed $b>1$ and a measurable function $\psi$,
define the linear operator $T_{\sigma}$ as
$$T_{\sigma}\psi(w)=\int_{\mathbb{D}}\frac{(1-|z|^{2})^{b-1}}{|1-\overline{z}w|^{b+\sigma}}\psi(z)dA(z), \ \ \sigma>0.$$

\begin{lemma}\label{lem2}
Let $0<p\leq 1$, $0<\beta\leq 1$, $\sigma>\max\{(3-p-2\beta)/2,
2-2\beta\}$, $p+2\beta>2$ and $\psi$ be a measureable function on
$\mathbb{D}$. If $|\psi(z)|^{2}(1-|z|^{2})^{p-2+2\beta}dA(z)$ is a
$(p+2-2\beta)$-Carleson measure, then
$|T_{\sigma}\psi(z)|^{2}(1-|z|^{2})^{2\sigma+p+2\beta-4}dA(z)$ is
also a $(p+2-2\beta)$-Carleson measure.
\end{lemma}

\begin{proof}
Given a subarc $I\subset T$ and any positive integer $n\leq
\log_{2}(1/|I|)$. We have

\begin{equation}
\begin{split}\nonumber
&\int_{S(I)}|T_{\sigma}\psi(w)|^{2}(1-|w|^{2})^{2\sigma+p+2\beta-4}dA(w)\\
\leq& \int_{S(I)}\left(\int_{\mathbb{D}}\frac{(1-|z|^{2})^{b-1}}{|1-\overline{z}w|^{b+\sigma}}|\psi(z)|
dA(z)\right)^{2}(1-|w|^{2})^{2\sigma+p+2\beta-4}dA(w)\\
=&\int_{S(I)}\left(\Big(\int_{S(2I)}+\int_{\mathbb{D}\backslash
S(2I)}\Big)\frac{(1-|z|^{2})^{b-1}}{|1-\overline{z}w|^{b+\sigma}}|\psi(z)|
dA(z)\right)^{2}(1-|w|^{2})^{2\sigma+p+2\beta-4}dA(w)\\
\leq&
2\int_{S(I)}\left(\int_{S(2I)}\frac{(1-|z|^{2})^{b-1}(1-|w|^{2})^{\sigma+\frac{p}{2}+\beta-2}}{|1-\overline{z}w|^{b+\sigma}}
|\psi(z)|dA(z)\right)^{2}dA(w)
\\
+&2\int_{S(I)}\left(\int_{\mathbb{D}\backslash S(2I)}\frac{(1-|z|^{2})^{b-1}(1-|w|^{2})^{\sigma+\frac{p}{2}+\beta-2}}
{|1-\overline{z}w|^{b+\sigma}}|\psi(z)|dA(z)\right)^{2}dA(w)\\
=:& M_{1}+M_{2}.
\end{split}
\end{equation}
Let
$$K(w,z)=\frac{(1-|z|^{2})^{b-\frac{p}{2}-\beta}(1-|w|^{2})^{\sigma+\frac{p}{2}+\beta-2}}{|1-\overline{z}w|^{b+\sigma}}.$$
Consider the operator
$$Bf(w)=\int_{\mathbb{D}}f(z)K(w,z)dA(z), \ f\in L^{2}(\mathbb{D}).$$
For $z\in \mathbb{D}$, using Lemma 3.10 of \cite{zhu}, we have
$$\int_{\mathbb{D}}K(w,z)(1-|w|^{2})^{-1/2}dA(w)\lesssim(1-|z|^{2})^{-1/2}$$
and
$$\int_{\mathbb{D}}K(w,z)(1-|z|^{2})^{-1/2}dA(z)\lesssim(1-|w|^{2})^{-1/2}.$$
By Schur's lemma (\cite[Corollary 3.7]{zhu}), $B$ is bounded on
$L^{2}(\mathbb{D})$.

Take
$$g(z)=(1-|z|^{2})^{\frac{p}{2}+\beta-1}|\psi(z)|\chi_{S(2I)}(z).$$
Since $|\psi(z)|^{2}(1-|z|^{2})^{p+2\beta-2}dA(z)$ is a
$(p+2-2\beta)$-Carleson measure, then $g\in L^{2}(\mathbb{D})$ and
$\|g\|_{L^{2}}^{2}\lesssim |I|^{p+2-2\beta}$. Thus,
\begin{equation}\nonumber
\begin{split}
M_{1}&\lesssim
\int_{\mathbb{D}}\Big|\int_{\mathbb{D}}K(w,z)g(z)dA(z)
\Big|^{2}dA(w)
\\&=\|B(g)\|_{L^{2}}^{2}\lesssim \|g\|_{L^{2}}^{2}\lesssim
|I|^{p+2-2\beta}.
\end{split}
\end{equation}

Write
$$D\backslash
S(2I)=\bigcup\limits_{n=1}^{\infty}S(2^{n+1}I)\backslash S(2^{n}I)=:
\bigcup\limits_{n=1}^{\infty} \Delta_{n}.$$ For $n\geq 0$, the
following inequality is well-known ( \cite[p.232]{GJB} ):
$$|1-\overline{z}w|\geq C 2^{n}|I|,$$
 where $\ w\in S(I)\  \mbox{and}\ z\in S(2^{n+1}I)\backslash S(2^{n}I)$. Note that
$$\int_{S(2^{n}I)}(1-|w|^{2})^{a}dA(w)\lesssim (2^{n}|I|)^{a+2},\ n\geq 0, a>-1.$$
Using H\"{o}lder's inequality, we have
\begin{equation}\nonumber
\begin{split}
\int_{S(2^{n+1}I)}|\psi(z)|(1-|z|^{2})^{b-1}dA(z)&\leq \left(
\int_{S(2^{n+1}I)}|\psi(z)|^{2}(1-|z|^{2})^{p+2\beta-2}dA(z)\right)^{1/2}\\
&\times\left(\int_{S(2^{n+1}I)} (1-|z|^{2})^{2b-p-2\beta}dA(z)\right)^{1/2}\\
&\leq \left(\int_{S(2^{n+1}I)}|\psi(z)|^{2}(1-|z|^{2})^{p+2\beta-2}dA(z)\right)^{1/2}\\
&\times(2^{n+1}|I|)^{b-\beta+1-p/2}.
\end{split}
\end{equation}

We get
\begin{equation}\nonumber
\begin{split}
M_{2}&=2\int_{S(I)}\left(\sum_{n=1}^{\infty}\int_{\Delta_{n}}\frac{(1-|z|^{2})^{b-1}}{|1-\overline{z}w|^{b+\sigma}}|\psi(z)|dA(z)
\right)^{2}(1-|w|^{2})^{2\sigma+p+2\beta-4}dA(w)\\
& \lesssim \int_{S(I)}\left(\sum_{n=1}^{\infty}\frac{1}{(2^{n}|I|)^{b+\sigma}}\int_{S(2^{n+1}I)}|\psi(z)|(1-|z|^{2})^{b-1}dA(z)\right)^{2}
(1-|w|^{2})^{2\sigma+p+2\beta-4}dA(w)\\
&\lesssim
|I|^{2\sigma+p+2\beta-2}\left(\sum_{n=1}^{\infty}\frac{1}{(2^{n}|I|)^{b+\sigma}}\int_{S(2^{n+1}I)}|\psi(z)
|(1-|z|^{2})^{b-1}dA(z)\right)^{2}\\
&\lesssim |I|^{p-2\beta+2}\left(\sum_{n=1}^{\infty}\frac{1}{2^{n(\sigma+2\beta-2)}}\left(\frac{1}{(2^{n+1}|I|)^{p-2\beta+2}
}\int_{S(2^{n+1}I)}|\psi(z)|^{2}(1-|z|^{2})^{p+2\beta-2}dA(z)\right)^{1/2} \right)^{2}\\
&\lesssim |I|^{p-2\beta+2}\Big(\sum_{n=1}^{\infty}\frac{1}{2^{n(\sigma+2\beta-2)}}\Big)^{2}
\\&\lesssim
|I|^{p-2\beta+2},
\end{split}
\end{equation}
where we used the fact that
$|\psi(z)|^{2}(1-|z|^{2})^{p-2+2\beta}dA(z)$ is a
$(p+2-2\beta)$-Carleson measure and $\sigma>2-2\beta$. This
completes the proof.
\end{proof}

\begin{remark} If $\beta=1$, the above lemma is Theorem 3 of
\cite{WX}.
\end{remark}

Let $\Gamma$ be the gamma function and $[\nu]$ the smallest integer
bigger than or equal to $\nu$($\nu>0$). For fixed $b>1$, consider
the fractional derivative, called $\nu$-derivative of $f$,
$$f^{(\nu)}(z)=\frac{\Gamma(b+\nu)}{\Gamma(b)}\int_{\mathbb{D}}\frac{\overline{w}^{[\nu-1]}f'(w)}{(1-\overline{w}z)^{b+\nu}}(1-|w|^{2})^{b-1}dA(w).$$
If $\nu$ is a positive integer, then $\nu$-derivative is just the
usual $\nu$-th order derivative. But $f^{(\nu)}$ depends on $b$ if
$\nu$ is not an integer (see, for example, \cite{WX}).

\begin{theorem}\label{th2}
Suppose that $1/2<\beta\leq 1$, $2-2\beta<p\leq 1$ and
$\nu>\max\{(3-p-2\beta)/2, 2-2\beta\}$. Let $f\in H(\mathbb{D})$.
Then $f$ belongs to $Q_{p}^{\beta}(\mathbb{D})$ if and only if
$|f^{(\nu)}(z)|^{2}(1-|z|^{2})^{2\nu+p+2\beta-4}dA(z)$ is a
$(p+2-2\beta)$-Carleson measure.
\end{theorem}
\begin{proof}
At first, if $f\in Q_{p}^{\beta}(\mathbb{D})$, then
$|f'(z)|^{2}(1-|z|^{2})^{p-2+2\beta}$ is a $(p+2-2\beta)$-Carleson
measure. Note that $\nu>\max\{(3-p-2\beta)/2, 2-2\beta\}$ and
$$|f^{(\nu)}(z)|\leq \frac{\Gamma(b+\nu)}{\Gamma(b)}\int_{\mathbb{D}}\frac{(1-|w|^{2})^{b-1}}{|1-\overline{w}z|^{b+\nu}}|f'(w)|dA(w)
=\frac{\Gamma(b+\nu)}{\Gamma(b)}T_{\nu}|f'(z)|.$$ The desired result
follows from Lemma $\ref{lem2}$.

Conversely, if
$|f^{(\nu)}(z)|^{2}(1-|z|^{2})^{2\nu+p+2\beta-4}dA(z)$ is a
$(p+2-2\beta)$-Carleson measure, then
$$\int_{\mathbb{D}}|f^{(\nu)}(z)|^{2}(1-|z|^{2})^{2\nu+p+2\beta-4}dA(z)<\infty.$$
 Combining this with H\"{o}lder's inequality yields
\begin{equation}\nonumber
\begin{split}
\int_{\mathbb{D}}|f^{(\nu)}(z)|(1-|z|^{2})^{\nu}dA(z)&\leq\left(
\int_{\mathbb{D}}|f^{(\nu)}(z)|^{2}(1-|z|^{2})^{2\nu}dA(z)
\right)^{1/2}
\\&\leq\left(
\int_{\mathbb{D}}|f^{(\nu)}(z)|^{2}(1-|z|^{2})^{2\nu+p+2\beta-4}dA(z)
\right)^{1/2}\\&<\infty.
\end{split}
\end{equation}
For $m=[\nu-1]$, $b>1$ and $z\in \mathbb{D}$, we know
$$\int_{\mathbb{D}}\frac{(1-|w|^{2})^{b-2}}{|1-\overline{w}z|^{b+m+1}}dA(w)<\infty.$$
Let $\psi(z)=|f^{(\nu)}(z)|(1-|z|^{2})^{\nu-1}$. Then
$|\psi(z)|^{2}(1-|z|^{2})^{p-2+2\beta}dA(z)$ is a
$(p+2-2\beta)$-Carleson measure. We next prove that
$$|f^{m+1}(z)|\lesssim |T_{m+1}\psi(z)|.$$
In fact, for $\lambda>0$,
$$\frac{1}{(1-\overline{w}z)^{\lambda}}=\sum_{j=0}^{\infty}\frac{\Gamma(j+\lambda)}{\Gamma(j+1)\Gamma(\lambda)}(\overline{w}z)^{j}.$$
It is easy to check that
\begin{equation}\nonumber
\begin{split}
(z^{k})^{(\nu)}=
\begin{cases}0
& \mbox{if}\  k<[\nu-1]+1;\\
\frac{\Gamma(k+1)\Gamma(b+k+\nu-1-[\nu-1])}{\Gamma(b+k)\Gamma(k-[\nu-1])}z^{k-1-[\nu-1]} & \mbox{if}\ k\geq
[\nu-1]+1.
\end{cases}
\end{split}
\end{equation}
Let $f(z)=\sum\limits_{j=0}^{\infty}a_{j}z^{j}$. A direct
computation shows that
$f^{(\nu)}(z)=\sum\limits_{j=0}^{\infty}a_{j}^{(\nu)}z^{j}$, where
$$a_{j}^{(\nu)}=a_{j+m+1}\frac{\Gamma(j+b+\nu)\Gamma(j+m+2)}{\Gamma(j+1)\Gamma(j+m+b+1)}.$$

Now applying Parseval's theorem (\cite{RW}), we obtain
\begin{equation}\nonumber
\begin{split}
&\frac{\Gamma(b+m+1)}{\Gamma(b+\nu-1)}\int_{\mathbb{D}}\frac{(1-|w|^{2})^{b-2}}{(1-\overline{w}z)^{b+m+1}}
f^{(\nu)}(w)(1-|w|^{2})^{\nu}dA(w)\\
=&\frac{\Gamma(b+m+1)}{\Gamma(b+\nu-1)}\int_{\mathbb{D}}(1-|w|^{2})^{b+\nu-2}\Big(\sum_{k=0}^{\infty}
\frac{\Gamma(k+b+m+1)}{k!\Gamma(b+m+1)}\overline{w}^{k}z^{k}\Big)\Big(\sum_{j=0}^{\infty}a_{j}^{(\nu)}w^{j}\Big)dA(w)\\
=&\frac{\Gamma(b+m+1)}{\Gamma(b+\nu-1)}\sum_{k=0}^{\infty}a_{k}^{(\nu)}z^{k}\frac{\Gamma(k+b+m+1)}{k!\Gamma(b+m+1)}
\int_{\mathbb{D}}(1-|w|^{2})^{b+\nu-2}|w|^{2k}dA(w)\\
=&\sum_{k=0}^{\infty}\frac{\Gamma(k+m+2)}{\Gamma(k+1)}a_{k+m+1}z^{k}\\
=&f^{(m+1)}(z).
\end{split}
\end{equation}

Hence $|f^{(m+1)}|\lesssim |T_{m+1}\psi(z)|$. Using Lemma
$\ref{lem2}$ again yields that
$|f^{(m+1)}|^{2}(1-|z|^{2})^{2m+p+2\beta-2}dA(z)$ is a
$(p+2-2\beta)$-Carleson measure. So the desired result follows from
Theorem 4.33 of \cite{RJ}.
\end{proof}

Based on Theorem $\ref{th2}$, we obtain the main result of this
section.
\begin{theorem}\label{th3}
Let $0<p\leq 1$, $1/2<\beta\leq 1$ with $2\beta-p\geq 1$. Then $f\in
Q_{p}^{\beta}(\mathbb{D})$ if and only if
$f^{(\frac{3-p-2\beta}{2})}\in
\mathcal{L}^{2,p-2\beta+2}(\mathbb{D})$.
\end{theorem}
\begin{proof}
For $\nu>0$, a simple calculation shows that
$(f^{(\nu)})'=f^{(\nu+1)}$. If $f\in Q_{p}^{\beta}(\mathbb{D})$, by
Theorem $\ref{th2}$,
$$\sup_{I\subset \mathbb{T}}\frac{1}{|I|^{p+2-2\beta}}\int_{S(I)}|f^{(\frac{5-p-2\beta}{2})}(z)|^{2}(1-|z|^{2})dA(z)<\infty.$$
From Lemma $\ref{lem3}$, we get $f^{(\frac{3-p-2\beta}{2})}\in
\mathcal{L}^{2,p-2\beta+2}(\mathbb{D})$.

On the other hand, if $f^{(\frac{3-p-2\beta}{2})}\in
\mathcal{L}^{2,p-2\beta+2}$. Using Lemma $\ref{lem3}$ again, we know
that $|f^{(\frac{5-p-2\beta}{2})}(z)|^{2}(1-|z|^{2})dA(z)$ is a
$(p+2-2\beta)$-Carleson measure. Note that
$\frac{5-p-2\beta}{2}>\max\{\frac{3-p-2\beta}{2}, 2-2\beta\}$,
Theorem $\ref{th2}$ implies $f\in Q_{p}^{\beta}(\mathbb{D})$. We
finishe the proof.
\end{proof}

\begin{remark}
If $\beta=1$, Theorem $\ref{th3}$ yields that $f\in
Q_{p}(\mathbb{D})$ if and only if $ f^{(\frac{1-p}{2})}\in
\mathcal{L}^{2, p}(\mathbb{D})$. This result on the upper half plane
was obtained in \cite{WX}. The relationship between $Q_{K}$ and
Morrey type spaces is established in \cite{Wulan}. Moreover, Theorem
$\ref{th3}$ can be compared with Theorem 3.2 of \cite{Wulan}.
\end{remark}

\section{\bf Multiplier spaces of $Q_{p}^{\beta}(\mathbb{D})$}
In this section, we consider the multiplier spaces of
$Q^{\beta}_{p}(\mathbb{D})$.
\begin{definition}
Let $X$ be a Banach space of analytic functions on $\mathbb{D}$.
\begin{itemize}
\item[(1)] An
analytic function $g\in H(\mathbb{D})$ is called a pointwise
multiplier on $X$ if $gX\subset X$.
\item[(2)] The multiplier space of $X$, denoted by $\mathcal{M}(X)$, is the set of all pointwise
multipliers on $X$.\end{itemize}
\end{definition}
In the study of multiplier spaces, one useful tool is  the following
multiplication operator induced by $g$: $$ M_{g}f=gf, \text{ where
}f\in X.$$ In fact,  by the Closed Graph Theorem, it can be proved
that $M_{g}$ is a bounded linear operator on $X$ if and only if
$g\in \mathcal{M}(X)$. We refer the reader to \cite{BS} for more
information about pointwise multipliers.

To study $M_{g}$, we need the following integral operators. For any
$g\in H(\mathbb{D})$,
\begin{itemize}
\item[(1)] Volterra type operator $T_{g}$ induced by $g$ is defined as:
 $$T_{g}f(z)=\int_{0}^{z}f(w)g'(w)dw, \ \ f\in H(\mathbb{D}).$$

\item[(2)] The operator
$I_{g}$ related to $T_{g}$ is defined as:
$$I_{g}f(z)=\int_{0}^{z}f'(w)g(w)dw,\ \ f\in H(\mathbb{D});$$
\end{itemize}
 \begin{remark}
Generally, $g$ is assumed to be non-constant. In fact, if $g$ is a
constant $C$, then $I_{g}=C(f(z)-f(0))$ and $T_{g}=0$.
\end{remark}

It is easy to see that the multiplication operator $M_{g}$
can be decomposed as
$$M_{g}f(z)=f(0)g(0)+I_{g}f(z)+T_{g}f(z).$$
Hence the boundedness of $M_{g}$ can be deduced from that of $I_{g}$
and $T_{g}$. We next study the boundedness of the operators $I_{g}$
and $T_{g}$  on $Q_{p}^{\beta}(\mathbb{D})$. We need the following
lemma (\cite[Lemma 1]{ZR}).
\begin{lemma}\label{lema}
Suppose that $s>-1$, $r, t>0$ and $t<s+2<r$. We have
$$\int_{\mathbb{D}}\frac{(1-|z|^{2})^{s}}{|1-\overline{b}z|^{r}|1-\overline{a}z|^{t}}dA(z)\lesssim
\frac{1}{(1-|b|^{2})^{r-s-2}|1-\overline{a}b|^{t}},$$ where $a, b
\in \mathbb{D}$.
\end{lemma}

\begin{lemma}\label{lemb}
Let $0<p\leq 1$ and $1/2<\beta\leq 1$. For $b\in \mathbb{D}$, the function
$$f_{b}(z)=(1-|b|^{2})^{2-2\beta}(\sigma_{b}(z)-b)$$ belongs to
$Q_{p}^{\beta}(\mathbb{D})$. Moreover, $\|f_{b}\|_{Q_{p}^{\beta}(\mathbb{D})}\lesssim 1.$
\end{lemma}

\begin{proof}
Applying Lemma 5 for $t=p+2-2\beta$, $r=4$ and $s=p+2\beta-2$, we
get
\begin{equation}
\begin{split}\nonumber
&\sup_{a\in
\mathbb{D}}\int_{\mathbb{D}}|f_{b}'(z)|^{2}(1-|z|^{2})^{4\beta-4}(1-|\sigma_{a}(z)|^{2})^{p+2-2\beta}dA(z)\\
=&\sup_{a\in
\mathbb{D}}\int_{\mathbb{D}}\frac{(1-|b|^{2})^{6-4\beta}(1-|z|^{2})^{p+2\beta-2}(1-|a|^{2})^{p+2-2\beta}}
{|1-\overline{b}z|^{4}|1-\overline{a}z|^{2(p+2-2\beta)}}dA(z)\\
\lesssim& (1-|b|^{2})^{6-4\beta}\sup_{a\in
\mathbb{D}}\int_{\mathbb{D}}\frac{(1-|z|^{2})^{p+2\beta-2}}
{|1-\overline{b}z|^{4}|1-\overline{a}z|^{p+2-2\beta}}dA(z)
\\\lesssim &\frac{(1-|b|^{2})^{2-2\beta+p}}{|1-\overline{a}b|^{p+2-2\beta}}
\\\lesssim & 1.
\end{split}
\end{equation}
So $f_{b}(z)\in Q_{p}^{\beta}(\mathbb{D})$ and
$\|f_{b}(z)\|_{Q_{p}^{\beta}(\mathbb{D})}\lesssim 1$.
\end{proof}

\begin{theorem}\label{tha}
For $0<p\leq 1$ and $1/2<\beta\leq 1$. Let $g\in H(\mathbb{D})$.
Then $I_{g}$ is bounded on $Q_{p}^{\beta}(\mathbb{D})$ if and only
if $g\in H^{\infty}$. Moreover, $\|I_{g}\|\approx \sup\limits_{z\in
\mathbb{D}}|g(z)|$.
\end{theorem}
\begin{proof}
It is easy to see that
\begin{equation}
\begin{split}\nonumber
\|I_{g}f\|_{Q_{p}^{\beta}(\mathbb{D})}&=\sup_{a\in\mathbb
D}\left(\int_{\mathbb{D}}|f'(z)|^{2}|g(z)|^{2}(1-|z|^{2})^{4\beta-4}(1-|\sigma_{a}(z)|^{2})^{p+2-2\beta}dA(z)\right)^{1/2}
\\&\leq\sup_{z\in \mathbb{D}}|g(z)|\|f\|_{Q_{p}^{\beta}(\mathbb{D})}.
\end{split}
 \end{equation}
That is
 $$\|I_{g}\|\leq \sup\limits_{z\in \mathbb{D}}|g(z)|.$$

 On the other hand, consider the test function $f_{b}(z)$ as in Lemma $\ref{lemb}$.
 We get
\begin{equation}
\begin{split}\nonumber
\|I_{g}\|&\gtrsim
\|I_{g}f_{b}\|_{Q_{p}^{\beta}(\mathbb{D})}\\&=\sup_{a\in\mathbb
D}\left(\int_{\mathbb{D}}(1-|b|^{2})|^{4-4\beta}|\sigma_{b}'(z)|^{2}|g(z)|^{2}(1-|z|^{2})^{4\beta-4}
(1-|\sigma_{a}(z)|^{2})^{p+2-2\beta}dA(z)\right)^{1/2}
\\&\geq \left(\int_{\mathbb{D}}(1-|b|^{2})|^{4-4\beta}|g(\sigma_{b}(z))|^{2}(1-|\sigma_{b}(z)|^{2})^{4\beta-4}
(1-|z|^{2})^{p+2-2\beta}dA(z)\right)^{1/2}
\\&=\left(\int_{\mathbb{D}}|g(\sigma_{b}(z))|^{2}|1-\overline{b}z|^{2(4-4\beta)}(1-|z|^{2})^{p+2\beta-2}dA(z)\right)^{1/2}
\\&\gtrsim |g(b)|,
\end{split}
 \end{equation}
 where we use Lemma 4.12 of \cite{zhu} in the last inequality.
Since $b$ is arbitrary in $\mathbb{D}$,  we have
$$\|I_{g}\|\gtrsim
\sup\limits_{z\in \mathbb{D}}|g(z)|.$$ This completes the proof.
\end{proof}
\begin{theorem}\label{thb}
 For $0<p\leq 1$ and $1/2<\beta<1$. Let $g\in H(\mathbb{D})$. Then $T_{g}$ is bounded on
$Q_{p}^{\beta}(\mathbb{D})$ if and only if $g\in
Q_{p}^{\beta}(\mathbb{D})$. Moreover, $\|T_{g}\|\approx
\|g\|_{Q_{p}^{\beta}(\mathbb{D})}$.
\end{theorem}
\begin{proof}
Take $f(z)\equiv 1$ on $\mathbb{D}$. It is easy to verify that
$\|f\|_{Q_{p}^{\beta}(\mathbb{D})}=1$. So
$$\|T_{g}\|\geq \|T_{g}f\|_{Q_{p}^{\beta}(\mathbb{D})}=\|g\|_{Q_{p}^{\beta}(\mathbb{D})}.$$

Conversely, from the proof of Theorem 2.10 in \cite{ZR1}, we know
$$\sup_{z\in \mathbb{D}}(1-|z|^{2})^{2\beta-1}|f'(z)|\lesssim \|f\|_{Q_{p}^{\beta}(\mathbb{D})},
\ f\in Q_{p}^{\beta}(\mathbb{D}).$$ Hence
\begin{equation}
\begin{split}\nonumber
|f(z)-f(0)|&=\Big|\int_{0}^{1}f'(tz)zdt\Big|
\\&\leq\|f\|_{Q_{p}^{\beta}(\mathbb{D})}\int_{0}^{1}\frac{1}{(1-t)^{2\beta-1}}dt
\\&\lesssim
\|f\|_{Q_{p}^{\beta}(\mathbb{D})}.
\end{split}
 \end{equation}
Then
$$|f(z)|\lesssim \|f\|_{Q_{p}^{\beta}(\mathbb{D})}.$$
We get
\begin{equation}
\begin{split}\nonumber
\|T_{g}f\|_{Q_{p}^{\beta}(\mathbb{D})}&=\sup_{a\in\mathbb
D}\left(\int_{\mathbb{D}}|f(z)|^{2}|g'(z)|^{2}(1-|z|^{2})^{4\beta-4}(1-|\sigma_{a}(z)|^{2})^{p+2-2\beta}dA(z)\right)^{1/2}
\\&\lesssim\|f\|_{Q_{p}^{\beta}(\mathbb{D})}\|g\|_{Q_{p}^{\beta}(\mathbb{D})}.
\end{split}
 \end{equation}
 Thus $\|T_{g}\|\lesssim\|g\|_{Q_{p}^{\beta}(\mathbb{D})}$. The proof is completed.
\end{proof}

From Theorems $\ref{tha}$ and $\ref{thb}$, we get the main result of
this section.

\begin{theorem}\label{thc}
Let $0<p\leq 1$ and $1/2<\beta<1$. Then
$\mathcal{M}(Q_{p}^{\beta}(\mathbb{D}))=Q_{p}^{\beta}(\mathbb{D})$.
\end{theorem}
\begin{proof}
If $g\in Q_{p}^{\beta}(\mathbb{D})$, by the proof of Theorem
\ref{thb} that $g$ is bounded. It follows from Theorems \ref{tha}
and \ref{thb} that $M_{g}$ is bounded on
$Q_{p}^{\beta}(\mathbb{D})$. So $g\in
\mathcal{M}(Q_{p}^{\beta}(\mathbb{D}))$ and then
$Q_{p}^{\beta}(\mathbb{D})\subset
\mathcal{M}(Q_{p}^{\beta}(\mathbb{D}))$.

 If $g\in
\mathcal{M}(Q_{p}^{\beta}(\mathbb{D}))$, take $f=1\in
Q_{p}^{\beta}(\mathbb{D})$, then $g=gf\in Q_{p}^{\beta}(\mathbb{D})$
and $\mathcal{M}(Q_{p}^{\beta}(\mathbb{D})) \subset
Q_{p}^{\beta}(\mathbb{D})$.
\end{proof}


\begin{thebibliography}{9}

\bibitem{ASX} Aulaskari R., Stegenga D. A., Xiao J.:  Some subclasses
of BMOA and their characterization in terms of Carleson measures,
Rocky Mountain J. Math. \textbf{26}(2), 485-506 (1996)

\bibitem{AXZ}Aulaskari R., Xiao J., Zhao R.: On subspaces and
subsets of BMOA and UBC, Analysis \textbf{15}(2),
101-121 (1995)

\bibitem{AA}  Anderson A.:
 Some closed range integral operators on spaces of analytic
functions, Integr. Equ. Oper. Theory \textbf{69}
(1), 87-99 (2011)

\bibitem{AA1}Anderson A.: Multiplication and Integral Operators on Banach Spaces of Analytic
Functions, PhD Thesis, University of Hawaii, 2010.

\bibitem{BS}Brown L., Shields A. L.: Cyclic vectors in the Dirichlet
space, Trans. Amer. Math. Soc. \textbf{285} (1), 269-303  (1984)

\bibitem{CL} Carleson L.: Interpolations by bounded analytic functions
and the corona problem, Ann.  Math. \textbf{76}, 547-559 (1962)

\bibitem{DX} Dafni G.,  Xiao J.:  Some new tent spaces and duality
theorems for fractional Carleson measures and
$Q_{\alpha}(\mathbb{R}^n)$, J. Funct. Anal. \textbf{208} (
2), 377-422 (2004)

\bibitem{EJPX} Ess\'{e}n M., Janson S.,  Peng L., Xiao J.: $Q$ spaces of several real variables,
 Indiana Univ. Math. J. \textbf{49} (2), 575-615 (2000)

\bibitem{GJB} Garnett J. B.: Bounded Analytic Functions, Revised first
edition. Graduate Texts in Mathematics, 236. Springer, New York,
2007.

\bibitem{GD} Girela D.: Analytic functions of bounded mean
oscillation, Complex Function Spaces (Mekrij\"{a}rvi, 1999),
61-170, Univ. Joensuu Dept. Math. Rep. Ser., 4, Univ. Joensuu,
Joensuu, 2001.

\bibitem{DP}Duren P. L.: Theory of $H^p$ Spaces, Pure and Applied Mathematics, Vol. 38 Academic Press, New York-London 1970.


\bibitem{KT} Koch H., Tataru D.: Well-posedness for the
Navier-Stokes equations, Adv. Math. \textbf{157}(1),
22-35 (2001)

\bibitem{LZ} Li P., Zhai Z.: Well-posedness and regularity of generalized
 Navier-Stokes equations in some critical $Q$-spaces, J. Funct. Anal. \textbf{259}(10), 2457-2519 (2010)

\bibitem{LZ2}
 Li P., Zhai Z.: Riesz transforms on $Q$-type spaces with
application to quasi-geostrophic equation, {arXiv: 0907.0856v1
[math. AP], 5 Jul. 2009} ({To appear in Taiwanese Journal of
Mathematics}).

 \bibitem{NX} Nicolau A., Xiao J.: Bounded functions in M\"{o}bius invariant Dirichlet spaces,
  J. Funct. Anal. \textbf{150}(2), 383-425 (1997)

  \bibitem{RJ} R\"{a}tty\"{a} J.: On Some Complex Function Spaces and
  Classes, University of Joensuu, 2001, PhD Thesis.

  \bibitem{RW} Rudin W.: Principles of Mathematical
  Analysis, Third edition. International Series in Pure and Applied Mathematics.
McGraw-Hill Book Co., New York-Auckland-D\"{u}sseldorf, 1976.

 \bibitem{WX} Wu Z.,  Xie C.: $Q$ spaces and Morrey spaces, J. Funct. Anal. \textbf{201}(1), 282-297 (2003)

\bibitem{WZ} Wu Z.: A new characterization for Carleson measure and some
applications, Integr. Equ. Oper. Theory \textbf{71}, 161-180 (2011)

\bibitem{Wulan} Wulan H., Zhou J.: QK and Morrey type spaces, Ann. Acad. Sci. Fenn. Math. \textbf{38}, 193-207
(2013)

 \bibitem{XJ0} Xiao J.: Homothetic variant of fractional Sobolev space with application to Navier-Stokes system,
  Dyn. Partial Differ. Equ. \textbf{4}(3), 227-245 (2007)

\bibitem{XJ}  Xiao J.: The $Q_p$ Carleson measure problem, Adv. Math. \textbf{217}(5), 2075-2088 (2008)

\bibitem{XJ1}  Xiao J.: Some essential properties of
$Q_p(\partial\Delta)$-spaces, J. Fourier Anal. Appl.
\textbf{6}(3), 311-323 (2000)

\bibitem{XJ2} Xiao  J.: Holomorphic $Q$ Classes, Lecture Notes in Mathematics,
1767. Springer-Verlag, Berlin, 2001.

 \bibitem{XJ3}Xiao, J.: Geometric $Q_p$ Functions, Frontiers in Mathematics. Birkh\"{a}user
Verlag, Basel, 2006.


\bibitem{YY} Yang D., Yuan W.:  A new class of function spaces connecting
Triebel-Lizorkin spaces and $Q$ spaces, J. Funct. Anal.
\textbf{255}(10), 2760-2809 (2008)

\bibitem{YR} Yoneda R. Multiplication operators, integration operators and
companion operators on weighted Bloch space, Hokkaido Math. J.
\textbf{34}(1), 135-147 (2005)

\bibitem{ZR} Zhao R.: Distances from Bloch functions to some M\"{o}bius
invariant spaces, Ann. Acad. Sci. Fenn. Math. \textbf{33}(1), 303-313 (2008)

\bibitem{ZR1} Zhao R.: On a general family of function spaces,
Ann. Acad. Sci. Fenn. Math. Diss. No. \textbf{105}(1996).

\bibitem{zhu} Zhu K.: Operator Theory in Function Spaces, Second edition. Mathematical Surveys and Monographs,
 138. Amer. Math. Soc, Providence (2007).
\end{thebibliography}
\end{document}